\documentclass[12pt]{amsart}
\usepackage{amsfonts}
\usepackage{amsmath}
\usepackage[mathscr]{eucal}
\usepackage{eurosym}
\usepackage{amssymb}
\usepackage{amstext}
\usepackage{amsthm}
\usepackage{mathrsfs}
\usepackage{centernot}
\usepackage{enumitem}
\usepackage{graphicx, color}
\usepackage{caption}
\usepackage{subcaption}
\usepackage[a4paper, total={6in, 8in}]{geometry}

\numberwithin{equation}{section}
\newtheorem{theorem}{Theorem}[section]

\newtheorem{corollary}[theorem]{Corollary}

\newtheorem{example}[theorem]{Example}

\newtheorem{proposition}[theorem]{Proposition}
\newtheorem{remark}[theorem]{Remark}

\begin{document}
\title[Relations between Chebyshev, Fibonacci
and Lucas polynomials]{Relations between Chebyshev, Fibonacci
and Lucas polynomials via trigonometric sums}
\author[L. Smajlovi\'{c}]{Lejla Smajlovi\'{c}}
\address{School of Economics and Business\\
University of Sarajevo\\
Trg Oslobodjenja Alija Izetbegovi\'{c} 1\\
71 000 Sarajevo\\
Bosnia and Herzegovina}
\email{lejla.smajlovic@efsa.unsa.ba}
\author[Z. \v{S}abanac]{Zenan \v{S}abanac}
\address{Department of Mathematics and Computer Science\\
University of Sarajevo\\
Zmaja od Bosne 33-35\\
71 000 Sarajevo\\
Bosnia and Herzegovina}
\email{zsabanac@pmf.unsa.ba}
\author[L. \v{S}\'{c}eta]{Lamija \v{S}\'{c}eta}
\address{School of Economics and Business\\
University of Sarajevo\\
Trg Oslobodjenja Alija Izetbegovi\'{c} 1\\
71 000 Sarajevo\\
Bosnia and Herzegovina}
\email{lamija.sceta@efsa.unsa.ba}

\begin{abstract}
In this paper we derive some new identities involving the Fibonacci and
Lucas polynomials and the Chebyshev polynomials of the first and the second kind. Our starting point is a finite trigonometric sum which equals the resolvent kernel on the discrete circle with $m$ vertices and which can be evaluated in two different ways. An expression for this sum in terms of the Chebyshev polynomials was deduced in \cite{JKS} and the expression in terms of the Fibonacci and Lucas polynomials is deduced in this paper. As a consequence, we establish some further identities involving trigonometric sums
and Fibonacci, Lucas, Pell and Pell-Lucas polynomials and numbers, thus providing a "physical" interpretation for those identities. Moreover, the finite trigonometric sum of the type considered in this paper can be related to the effective resistance between any two vertices of the $N$-cycle graph with four nearest neighbors $C_{N}(1,2)$. This yields further identities involving Fibonacci numbers.
\end{abstract}

\maketitle

\section{Introduction and statement of main results}

The Fibonacci sequence $%
\left\{ F_{n}\right\} _{n=0}^{\infty }$ is defined by recursive relation%
\begin{equation*}
F_{0}=0,\text{ }F_{1}=1,\text{ }F_{n}=F_{n-1}+F_{n-2}\text{ for }n\geq 2,
\end{equation*}%
and the Lucas sequence $\left\{ L_{n}\right\} _{n=0}^{\infty }$ is defined by%
\begin{equation*}
L_{0}=2,\text{ }L_{1}=1,\text{ }L_{n}=L_{n-1}+L_{n-2}\text{ for }n\geq 2.
\end{equation*}

Both sequences arise in many applications across many mathematical and
non-mathematical disciplines, such as art, architecture, biology, chemistry,
chess, electrical engineering, music, origami, poetry, physics, physiology,
neurophysiology, and many others (see \cite{K1}). Recent applications include
algebraic coding theory, linear sequential circuits, quasicrystals,
phyllotaxies, biomathematics, and computer science (see, e.g. \cite{FSZ, HKN12, Rey}).

A vast collection of interesting identities and many other properties involving
Fibonacci and Lucas numbers, as well as other integer sequences, can be
found in \cite{K1, K2}. The original proofs of many of these identities
employ diverse techniques, including reciprocal sums, Binet-like formulas,
generating functions, tiling, combinatorial approaches, determinant and
matrix methods.

The Fibonacci polynomials are defined by the recurrence relation:
\begin{equation}\label{eq. recurr fib}
F_{0}(x)=0,\text{ }F_{1}(x)=1,\text{ }F_{n}(x)=xF_{n-1}(x)+F_{n-2}(x)\text{
for }n\geq 2\text{.}
\end{equation}
The Lucas polynomials satisfy the same recurrence relation, with a different initial condition:%
\begin{equation}\label{eq. recurr lucas}
L_{0}(x)=2\text{, }L_{1}(x)=x\text{, and }L_{n}(x)=xL_{n-1}(x)+L_{n-2}(x)%
\text{ for }n\geq 2\text{.}
\end{equation}
The Fibonacci and Lucas numbers are special values of the Fibonacci and Lucas polynomials: $F_{n}(1)=F_{n}$ and $L_{n}(1)=L_{n}$ for nonnegative integers $n$.

The Fibonacci and Lucas polynomials and their generalizations appear in different applications in various disciplines such as physics, biology, computer science, and statistics, see \cite{K1, K2}. For example, both fractional and higher-order derivatives can be easily calculated using Lucas and Fibonacci polynomial relations, see \cite{E-S23, KA23}. Lucas polynomials arise in numerical methods for solving differential and integral equations \cite{CSG15, Na17, Or17, AHNB21,NSA22}.
The expressions for the Fibonacci and Lucas polynomials in terms of the Chebyshev polynomials have been derived in \cite{A-EYE-SS17, A-EY18, KKJD19} and corollaries of \cite[Theorem 3]{A-EN23}, however, the coefficients are rather complicated as they are expressed in terms of the hypergeometric functions. This is expected, because the recurrence relations \eqref{eq. recurr fib} and \eqref{eq. recurr lucas} indicate that the Chebyshev polynomials of the first and the second kind are equivalent to the Fibonacci and Lucas polynomials, when viewed as the generalized Fibonacci polynomial sequences, see \cite{FMcAM18}.

Many identities involving the Fibonacci and Lucas polynomials have been discovered so far. A very long list of those identities is provided in \cite{K2}, see also \cite{YZ02, WZ13, FG21}. The aim of this paper is to provide a "physical" interpretation behind the new and the known identities. Namely, the identities derived in this paper stem from different evaluations of the same trigonometric sum, which possesses a physical interpretation. Let us describe our two main results in a greater detail.

The starting point for our first main result is an observation that the zeros of Fibonacci and Lucas polynomials can be expressed in terms of basic trigonometric functions (see formula \eqref{roots} below). Therefore, the logarithmic derivative of those polynomials is equal to a certain finite trigonometric sum. Curiously, this sum, after some trivial manipulations can be identified with the multiple of a special value of the resolvent kernel on the discrete circle $X_m$ with $m$ vertices, twisted by an additive character. By using results of \cite{JKS} in which this sum is evaluated in terms of the Chebyshev polynomials, we deduce a new identity relating the Fibonacci and Lucas polynomials with the Chebyshev polynomials of the first and the second kind, as stated in the following theorem.
\begin{theorem}
\label{th1} The following identities hold true for all positive integers $n$:
\begin{equation}
\frac{L_{n}\left( 2x\right) }{4x\left( x^{2}+1\right) F_{n}\left( 2x\right) }%
=\frac{U_{n-1}\left( 2x^{2}+1\right) }{T_{n}\left( 2x^{2}+1\right) +\left(
-1\right) ^{n-1}}  \label{F1}
\end{equation}
and
\begin{equation}
\frac{F_{n}\left( 2x\right) }{xL_{n}\left( 2x\right) }=\frac{U_{n-1}\left(
2x^{2}+1\right) }{T_{n}\left( 2x^{2}+1\right) +\left( -1\right) ^{n}}.
\label{F2}
\end{equation}
\end{theorem}

As a corollary, we derive identities relating certain finite trigonometric
sums and the Fibonacci, Lucas, Pell and Pell-Lucas polynomials. Moreover, by using the relation \eqref{eq. relation luc fib to cheb} below between the Chebyshev and the Lucas and Fibonacci polynomials, from Theorem \ref{th1} we immediately derive the following functional relation between Chebyshev polynomials evaluated at different arguments.
\begin{corollary}
  For any positive integer $n$ and all $x\in \mathbb{R}$ we have
  $$
  T_n(x)(1-T_n(2x^2-1))=x(x^2-1)U_{n-1}(x)U_{n-1}(2x^2-1).
  $$
\end{corollary}

The second main result, given in Theorem \ref{th2} below is also deduced by starting with a certain finite trigonometric sum which possesses a physical interpretation; namely it can be viewed as the effective resistance between two nodes in a certain special resistor network (see, e.g. \cite{NC1, NC2, Wu}). This sum is evaluated in \cite{NC2} in terms of the Fibonacci numbers. A different evaluation of this sum in terms of special values of the Chebyshev polynomials yields the following theorem.
\begin{theorem}
\label{th2} For a positive integer $N$, let
\begin{equation}
C_{N}=\left(1+\left( \frac{3-\sqrt{5}}{2}\right) ^{N}\right)\left(1-\left( \frac{3-%
\sqrt{5}}{2}\right) ^{N}\right)^{-1}.  \label{C_N}
\end{equation}
The following identity holds true
\begin{equation}\label{newid1}
\frac{F_{2N-2\ell }+(-1)^N F_{2\ell }+\left( -1\right) ^{\ell +1}F_{2N}}{\frac{1}{2}%
L_{2N}-(-1)^N} =\sqrt{5}F_{\ell }^{2}C_{N}^{(-1)^N} - F_{2\ell }.
\end{equation}%
\end{theorem}

As corollaries we derive further identities involving Fibonacci and Lucas numbers as well as Fibonacci hyperbolic sine and cosine functions.

The paper is organized as follows: In Section 2 we present some
known relationships between the Chebyshev, Fibonacci, Lucas, Pell and Pell-Lucas
polynomials and the Fibonacci, Lucas, Pell and Pell-Lucas numbers, as well as
the Fibonacci hyperbolic sine and cosine functions. In Section 3 we recall evaluations of certain trigonometric sums in terms of the Chebyshev polynomials and the Fibonacci
and Lucas numbers, stemming from the analysis of the resolvent kernel and the resistance distance on circular graphs. In Section 4 we prove our two main theorems, while in Section 5 we establish some further identities involving the Fibonacci, Lucas, Pell and Pell-Lucas polynomials and numbers as well as the Fibonacci hyperbolic sine and cosine functions and the Bejaia and Pisa numbers.

\section{Preliminaries}

In this section we recall some well-known properties of the Chebyshev polynomials of the first and the second kind, the Fibonacci and Lucas polynomials and the Fibonacci hyperbolic sine and cosine functions.

\subsection{Chebyshev polynomials and their relations with Fibonacci and
Lucas numbers}

The Chebyshev polynomials are sequences of orthogonal polynomials
related to the cosine and sine functions. The Chebyshev polynomials of the
first kind, denoted by $T_{n}$, $n\in \mathbb{N}\cup \{0\}$ are the unique polynomials satisfying
\begin{equation*}
T_{n}\left( \cos \theta \right) =\cos \left( n\theta \right),\,\, n=0,1,2,\ldots.
\end{equation*}
As functions of a complex variable $x$ they can be defined as
\begin{equation*}
T_{n}\left( x\right) = \frac{1}{2}\left(\left(x-\sqrt{x^2-1}\right)^n + \left(x+\sqrt{x^2-1}\right)^n \right),
\end{equation*}
where, as is standard, we assume that the branch of the square root is principal.

The Chebyshev polynomials of the second kind, denoted by $U_{n}$, $n\in \mathbb{N}\cup \{0\}$ are defined as the unique polynomials satisfying
\begin{equation*}
U_{n}\left( \cos \theta \right) \sin \theta =\sin \left( \left( n+1\right)
\theta \right),\,\, n=0,1,2,\ldots.
\end{equation*}
A very good reference, containing information about numerous applications of the Chebyshev polynomials in approximation theory, algebra, number theory, ergodic theory, and numerical analysis is \cite{Riv}.

The Chebyshev polynomials possess the following well-known symmetry
properties
\begin{eqnarray}
T_{n}\left( -x\right) &=&\left( -1\right) ^{n}T_{n}\left( x\right) ,
\label{sym1} \\
U_{n}\left( -x\right) &=&\left( -1\right) ^{n}U_{n}\left( x\right) .
\label{sym2}
\end{eqnarray}

There are various relations between the values of Chebyshev polynomials for
specific values of the variable and the Fibonacci and Lucas numbers (see e.g.
\cite[Section 2]{Cas}, \cite[Remark 9]{JG} or \cite{Zh02}). Below, we list some of them:
\begin{align}
T_{n}\left( \frac{3}{2}\right) & =\frac{L_{2n}}{2};\text{ \ \ }U_{n}\left(
\frac{3}{2}\right) =F_{2n+2};  \label{CFL1} \\
T_{n}\left( \frac{7}{2}\right) & =\frac{L_{4n}}{2};\text{ \ \ }U_{n}\left(
\frac{7}{2}\right) =\frac{F_{4n+4}}{3};  \label{CFL2} \\
T_{n}\left( \frac{\sqrt{5}}{2}\right) & =\left\{
\begin{array}{ll}
\frac{L_{n}}{2}, & \text{if }n\text{ is even;} \\
\frac{\sqrt{5}F_{n}}{2}, & \text{if }n\text{ is odd;}%
\end{array}%
\right. \text{ \ \ }U_{n}\left( \frac{\sqrt{5}}{2}\right) =\left\{
\begin{array}{ll}
L_{n+1}, & \text{if }n\text{ is even;} \\
\sqrt{5}F_{n+1}, & \text{if }n\text{ is odd;}%
\end{array}%
\right. \text{ }  \label{CFL3}
\end{align}
\begin{align}
T_{n}\left( \sqrt{5}\right) & =\left\{
\begin{array}{ll}
\frac{L_{3n}}{2}, & \text{if }n\text{ is even;} \\
\frac{\sqrt{5}F_{3n}}{2}, & \text{if }n\text{ is odd;}%
\end{array}%
\right. \text{ \ \ }U_{n}\left( \sqrt{5}\right) =\left\{
\begin{array}{ll}
\frac{L_{3n+3}}{4}, & \text{if }n\text{ is even;} \\
\frac{\sqrt{5}F_{3n+3}}{4}, & \text{if }n\text{ is odd.}%
\end{array}%
\right.  \label{CFL4}
\end{align}

By substituting $m$ by $2m$ and $x=\frac{\sqrt{5}}{2}$ in the equality $T_{n}\left( T_{m}\left( x\right) \right) =T_{nm}\left( x\right)$,
and using above identities, we obtain
\begin{equation}
T_{n}\left( \frac{1}{2}L_{2m}\right) =\frac{1}{2}L_{2nm}.  \label{chebluc}
\end{equation}
For the Chebyshev polynomials of the second kind we have (see \cite[Section
2]{Cas})
\begin{equation}
U_{n-1}\left( \frac{1}{2}L_{2m}\right) =\frac{F_{2nm}}{F_{2m}}.
\label{chebfib}
\end{equation}

\subsection{Fibonacci and Lucas polynomials}

The Fibonacci and Lucas polynomials are defined by the recurrence relations \eqref{eq. recurr fib} and \eqref{eq. recurr lucas}. Hoggatt and Bicknell \cite{HB} proved that the roots of the Fibonacci and Lucas polynomials are given as follows:
\begin{equation}
\begin{array}{lcll}
F_{2n}\left( x\right) =0\text{ } & \text{at } & x=\pm 2i\sin \frac{k\pi }{2n}%
, & k=0,1,\ldots ,n-1, \\
F_{2n+1}\left( x\right) =0\text{ } & \text{at } & x=\pm 2i\sin \frac{\left(
k+1/2\right) \pi }{2n+1}, & k=0,1,\ldots ,n-1, \\
L_{2n}\left( x\right) =0\text{ } & \text{at } & x=\pm 2i\sin \frac{\left(
k+1/2\right) \pi }{2n}, & k=0,1,\ldots ,n-1, \\
L_{2n+1}\left( x\right) =0\text{ } & \text{at } & x=\pm 2i\sin \frac{k\pi }{%
2n+1}, & k=0,1,\ldots ,n.%
\end{array}
\label{roots}
\end{equation}

The first derivative of the Lucas polynomial can be expressed as
\begin{equation}\label{eq. deriv lucas polyn}
L_{n}^{\prime }\left( x\right)
=nF_{n}\left( x\right),
\end{equation}
see \cite[formula (31.22) on p. 15]{K2}). The first derivative of the Fibonacci polynomial can easily be derived using the formula at the bottom of p. 15 of \cite{K2} to get
\begin{equation}
F_{n}^{\prime }\left( x\right) =\frac{nL_{n}\left( x\right) -xF_{n}\left(
x\right) }{x^{2}+4} . \label{fibderivative}
\end{equation}

An in-depth study of the Fibonacci and Lucas polynomials can be found in
\cite[Chapters 31 and 32]{K2}. Moreover, in \cite[Chapter 41]{K2}, it is proved that for $n\geq 1$ we have the following identities, relating the Lucas and Fibonacci polynomials to the Chebyshev polynomials of the first and the second kind:
\begin{equation}\label{eq. relation luc fib to cheb}
L_n(2x)=2(-i)^n T_n(ix) \quad \text{and} \quad F_n(2x)= -(-i)^{n-1}U_{n-1}(ix).
\end{equation}

Using these polynomials, two additional classes of polynomials, the Pell polynomials $P_{n}\left( x\right) $ and the Pell-Lucas polynomials $Q_{n}\left( x\right) $ are defined:
\begin{equation}
P_{n}\left( x\right) =F_{n}(2x)= -(-i)^{n-1}U_{n-1}(ix) \text{  and  }Q_{n}\left( x\right) =L_{n}(2x) = 2(-i)^n T_n(ix).
\label{PPL}
\end{equation}%
The corresponding Pell numbers $P_{n}$ and the Pell-Lucas numbers $Q_{n}$ are
given by:
\begin{equation*}
P_{n}=P_{n}(1)=F_{n}(2) =-(-i)^{n-1}U_{n-1}(i)
\end{equation*}
and
\begin{equation*}
Q_{n}=\frac{1}{2}Q_{n}(1)=\frac{1}{2}  %
L_{n}(2)= (-i)^n T_n(i).
\end{equation*}%
For further details and properties of Pell and Pell-Lucas polynomials and
numbers, we refer an interested reader to \cite[Section 31.6]{K2}.

The values of Chebyshev polynomials at $x=3$ are also related to the Pell and
Pell-Lucas numbers by (see \cite[Exercises 65 and 66 on p. 589]{K2}):
\begin{equation}
U_{n-1}\left( 3\right) =\frac{1}{2}P_{2n}\text{ and }T_{n}\left( 3\right)
=Q_{2n}.  \label{CPL1}
\end{equation}

\subsection{Fibonacci hyperbolic sine and cosine functions}

For $x\in (-\infty ,+\infty )$, Stakhov and Tkachenko \cite{ST93} (see also \cite{Trz, SR05}) introduced the Fibonacci hyperbolic sine and cosine functions by analogy to the classic hyperbolic
functions:
\begin{equation*}
\mathrm{sFh}\left( x\right) =\frac{\psi ^{x}-\psi ^{-x}}{\sqrt{5}}\quad \text{ and
} \quad \mathrm{cFh}\left( x\right) =\frac{\psi ^{\left( x+\frac{1}{2}\right)
}+\psi ^{-\left( x+\frac{1}{2}\right) }}{\sqrt{5}},
\end{equation*}%
where $\psi =1+\phi =\frac{3+\sqrt{5}}{2}$, and $\phi =\frac{1+\sqrt{5}}{2}$
denotes the golden ratio. Using the well-known identity $\phi ^{2}=1+\phi $,
one can obtain%
\begin{equation*}
\mathrm{sFh}\left( x\right) =\frac{\phi ^{2x}-\phi ^{-2x}}{\sqrt{5}}\quad \text{
and } \quad \mathrm{cFh}\left( x\right) =\frac{\phi ^{\left( 2x+1\right) }+\phi
^{-\left( 2x+1\right) }}{\sqrt{5}}.
\end{equation*}

For $k\in\mathbb{Z}$, the Fibonacci hyperbolic sine and cosine functions are related to the classical
hyperbolic sine and cosine functions by (see \cite[relation (12)]%
{Trz}):%
\begin{equation}
\mathrm{sFh}\left( k\right) =\frac{2}{\sqrt{5}}\sinh \left( 2k\log \phi
\right) \text{   and   } \mathrm{cFh}\left( k\right) =\frac{2}{\sqrt{5}}\cosh
\left( \left( 2k+1\right) \log \phi \right) .  \label{Fh1}
\end{equation}

For nonnegative integers $k$, functions $\mathrm{sFh}\left( k\right) $ and $%
\mathrm{cFh}\left( k\right) $ can be expressed in terms of the corresponding
elements of the Fibonacci sequence (see \cite[relation (6)]{Trz}):%
\begin{equation}
\mathrm{sFh}\left( k\right) =F_{2k} \quad\text{ and }\quad\mathrm{cFh}\left( k\right)
=F_{2k+1}\text{.}  \label{Fh2}
\end{equation}

\section{Closed evaluation of some finite trigonometric sums}

In this section we recall results of \cite{JKS} and \cite{NC2} in which certain finite trigonometric sums are expressed as rational functions of Chebyshev polynomials and in terms of Fibonacci numbers.

\subsection{The resolvent kernel on the discrete circle}

Jorgenson, Karlsson and Smajlovi\'{c} in \cite{JKS} computed the resolvent
kernel $G_{X_{m},\chi _{\beta }}(x,y;s)$ on the discrete circle $X_m$ with $m$
vertices using two different methods. The first employs the spectral
expansion of the combinatorial Laplacian which acts on the space of
functions on $X_{m}$ that are twisted by a character $\chi _{\beta }$,
resulting in its expression as a finite trigonometric sum, while
the second one utilizes the perspective of $X_{m}$ as a quotient space of $%
\mathbb{Z}$, thus expressing the resolvent kernel as a rational function
involving Chebyshev polynomials. It has been shown that (see \cite[relations
(41) and (42)]{JKS})
\begin{equation}
\frac{1}{m}\sum_{j=0}^{m-1}\frac{e^{2\pi i\frac{j\ell }{m}}}{s+2\sin
^{2}\left( \pi \frac{j+\beta }{m}\right) }=e^{2\pi i\frac{\beta \ell }{m}%
}\cdot \frac{U_{m-\ell -1}\left( s+1\right) +e^{2\pi i\beta }U_{\ell
-1}\left( s+1\right) }{T_{m}\left( s+1\right) -\cos 2\pi \beta },
\label{sum0}
\end{equation}%
for $\ell \in \left\{ 0,1,\ldots ,m-1\right\} $ and all complex $s$.
Setting $\beta =0$ in \eqref{sum0}, we derive that%
\begin{equation}
\frac{1}{m}\sum_{j=0}^{m-1}\frac{e^{2\pi i\frac{j\ell }{m}}}{s+2\sin ^{2}%
\frac{\pi j}{m}}=\frac{U_{m-\ell -1}\left( s+1\right) +U_{\ell -1}\left(
s+1\right) }{T_{m}\left( s+1\right) -1}.  \label{sum1}
\end{equation}

\subsection{Effective resistance}

Computation of resistance in electric circuits, specifically the effective resistance between two nodes in a resistor network, can be carried out using a variety of methods. Traditional analytic methods, such as Kirchhoff's laws, are theoretically applicable to any network but become impractical for larger networks (with more nodes) due to the rapidly increasing number of equations \cite{NC2}. For this reason, computations of resistance in infinite networks have been conducted using the Green's function \cite{Cs}.

In 2004 Wu \cite{Wu} introduced an approach to calculate the two-point resistance (effective resistance) in a finite resistor network in terms of the eigenvalues and eigenvectors of the Laplacian matrix, and derived explicit formulas for two-point resistances in regular lattices of one, two, and three dimensions under various boundary conditions. The following summation identity was deduced in \cite[formula (62)]{Wu}
\begin{equation}
\frac{1}{m}\sum_{j=0}^{m-1}\frac{\cos \left( 2\ell \frac{j\pi }{m}\right) }{%
\cosh \lambda -\cos \left( 2\frac{j\pi }{m}\right) }=\frac{\cosh \left(
\frac{m}{2}-\ell \right) \lambda }{\sinh \lambda \sinh \frac{m\lambda }{2}},
\label{wu}
\end{equation}%
for $\lambda $ real, $m\in
\mathbb{N}$ and $\ell \in \left\{ 0,1,\ldots ,m-1\right\} $.

Chair \cite[relations (4), (5), (19), and (22)]{NC2} calculated the exact expression for the effective resistance between any two vertices of the $N$-cycle graph with four nearest neighbors $C_{N}(1,2)$ in terms of the
effective resistance of the $N$-cycle graph $C_{N}$, the square of the Fibonacci numbers, and the bisected Fibonacci numbers:%
\begin{eqnarray}
R\left( \ell \right) &=&\frac{1}{5}\ell \left( 1-\frac{\ell }{N}\right) +%
\frac{4}{5N}\sum\limits_{j=1}^{N-1}\frac{\sin ^{2}\frac{j\ell \pi }{N}}{%
5\left( 1-\frac{4}{5}\sin ^{2}\frac{j\pi }{N}\right) },  \label{R1} \\
R\left( \ell \right) &=&\frac{1}{5}\ell \left( 1-\frac{\ell }{N}\right)
+(-1)^{\ell +1}\frac{F_{\ell }^{2}}{\sqrt{5}}C_{N}^{(-1)^N}+(-1)^{\ell }\frac{%
F_{2\ell }}{5}, \label{R2}
\end{eqnarray}%
where $C_{N}$ is given by \eqref{C_N}.

\section{Proofs of main results}

In this section, we prove Theorems \ref{th1} and \ref{th2}.

\subsection{Proof of the Theorem \ref{th1}}

First we will prove \eqref{F1}. If $n=2m$ is even, using \eqref{roots}, we write $F_{2m}$ as the product
\begin{equation*}
F_{2m}\left( x\right) =x\prod\limits_{j=1}^{m-1}\left( x^{2}+4\sin ^{2}\frac{%
j\pi }{2m}\right),
\end{equation*}
which, after taking the logarithmic derivative yields to
\begin{equation*}
\frac{F_{2m}^{\prime }\left( x\right) }{xF_{2m}\left( x\right) }%
=\sum\limits_{j=1}^{m-1}\frac{1}{\frac{x^{2}}{2}+2\sin ^{2}\frac{j\pi }{2m}}+%
\frac{1}{x^{2}}.
\end{equation*}
Using the formula \eqref{fibderivative} for the derivative of the Fibonacci polynomial, we get
\begin{equation}
\frac{2mL_{2m}\left( x\right) -xF_{2m}\left( x\right) }{x\left(
x^{2}+4\right) F_{2m}\left( x\right) }=\sum\limits_{j=1}^{m-1}\frac{1}{\frac{%
x^{2}}{2}+2\sin ^{2}\frac{j\pi }{2m}}+\frac{1}{x^{2}}.  \label{LF1}
\end{equation}

By a simple change of variables, it is immediate to deduce that
\begin{equation*}
\sum_{j=1}^{m-1}\frac{1}{\frac{x^{2}}{2}+2\sin ^{2}\frac{j\pi }{2m}}%
=\sum_{j=m+1}^{2m-1}\frac{1}{\frac{x^{2}}{2}+2\sin ^{2}\frac{j\pi }{2m}},
\label{symmetry1}
\end{equation*}
hence,%
\begin{equation*}
\frac{1}{2}\sum_{j=0}^{2m-1}\frac{1}{\frac{x^{2}}{2}+2\sin ^{2}\frac{j\pi }{%
2m}}=\sum_{j=1}^{m-1}\frac{1}{\frac{x^{2}}{2}+2\sin ^{2}\frac{j\pi }{2m}}+%
\frac{1}{x^{2}}+\frac{1}{x^{2}+4}.
\end{equation*}

Combining the last formula and \eqref{LF1}, we derive
\begin{equation}
\begin{aligned}\label{LF2}
\frac{1}{2}\sum_{j=0}^{2m-1}\frac{1}{\frac{x^{2}}{2}+2\sin ^{2}\frac{j\pi }{%
2m}}&=\frac{2mL_{2m}\left( x\right)-xF_{2m}\left( x\right) }{x\left(
x^{2}+4\right) F_{2m}\left( x\right) }+\frac{1}{x^{2}+4}\\
&=\frac{2mL_{2m}\left( x\right) }{x\left( x^{2}+4\right) F_{2m}\left( x\right) }.
\end{aligned}
\end{equation}

Setting $s=\frac{x^{2}}{2}$, $\ell =0$ and substituting $m$ by $2m$ in %
\eqref{sum1}, we obtain
\begin{equation}
\frac{1}{2}\sum_{j=0}^{2m-1}\frac{1}{\frac{x^{2}}{2}+2\sin ^{2}\frac{\pi j}{%
2m}}=\frac{mU_{2m-1}\left( \frac{x^{2}}{2}+1\right) }{T_{2m}\left( \frac{%
x^{2}}{2}+1\right) -1}.  \label{UT1}
\end{equation}

From \eqref{LF2} and \eqref{UT1} it follows that
\begin{equation*}
\frac{2L_{2m}\left( x\right) }{x\left( x^{2}+4\right) F_{2m}\left( x\right) }%
=\frac{U_{2m-1}\left( \frac{x^{2}}{2}+1\right) }{T_{2m}\left( \frac{x^{2}}{2}%
+1\right) -1}.
\end{equation*}
Replacing $x$ by $2x$ in the last formula proves \eqref{F1} for even $n$.

For odd numbers $n=2m+1$, using \eqref{roots}, we can write $F_n(x)$ as the following product
\begin{equation*}
F_{2m+1}\left( x\right) =\prod\limits_{j=0}^{m-1}\left( x^{2}+4\sin ^{2}%
\frac{\left( j+1/2\right) \pi }{2m+1}\right).
\end{equation*}
Upon taking the logarithmic derivative, it is straightforward to get
\begin{equation*}
\frac{\left( 2m+1\right) L_{2m+1}\left( x\right) -xF_{2m+1}\left( x\right) }{%
x\left( x^{2}+4\right) F_{2m+1}\left( x\right) }=\frac{1}{2}%
\sum\limits_{j=0}^{2m}\frac{1}{\frac{x^{2}}{2}+2\sin ^{2}\frac{\left(
j+1/2\right) \pi }{2m+1}}-\frac{1}{x^{2}+4},
\end{equation*}
or, equivalently
\begin{equation*}
\frac{\left( 2m+1\right) L_{2m+1}\left( x\right) }{x\left( x^{2}+4\right)
F_{2m+1}\left( x\right) }=\frac{1}{2}\sum\limits_{j=0}^{2m}\frac{1}{\frac{%
x^{2}}{2}+2\sin ^{2}\frac{\left( j+1/2\right) \pi }{2m+1}}.
\end{equation*}

On the other hand, using \eqref{sum0} with $s=\frac{x^{2}}{2}$, $\ell =0$,
and $\beta =\frac{1}{2}$, the right-hand side of the above display can be written as%
\begin{equation*}
\sum\limits_{j=0}^{2m}\frac{1}{\frac{x^{2}}{2}+2\sin ^{2}\frac{\left(
j+1/2\right) \pi }{2m+1}}=\frac{\left( 2m+1\right) U_{2m}\left( \frac{x^{2}}{%
2}+1\right) }{T_{2m+1}\left( \frac{x^{2}}{2}+1\right) +1}.
\end{equation*}

Combining the last two equations we get
\begin{equation*}
\frac{2L_{2m+1}\left( x\right) }{x\left( x^{2}+4\right) F_{2m+1}\left(
x\right) }=\frac{U_{2m}\left( \frac{x^{2}}{2}+1\right) }{T_{2m+1}\left(
\frac{x^{2}}{2}+1\right) +1}.
\end{equation*}
Substituting $x$ by $2x$ proves \eqref{F1} for $n$ odd.

Now we prove \eqref{F2}. For $n=2m$ we have%
\begin{equation*}
L_{2m}\left( x\right) =\prod\limits_{j=0}^{m-1}\left( x^{2}+4\sin ^{2}\frac{%
\left( j+1/2\right) \pi }{2m}\right) .
\end{equation*}

Taking the logarithmic derivative of the above display, and using \eqref{eq. deriv lucas polyn} we obtain%
\begin{equation*}
\frac{2mF_{2m}\left( x\right) }{xL_{2m}\left( x\right) }=\sum%
\limits_{j=0}^{m-1}\frac{1}{\frac{x^{2}}{2}+2\sin ^{2}\frac{\left(
j+1/2\right) \pi }{2m}}.
\end{equation*}

Using the symmetry property
\begin{equation*}
\sum\limits_{j=0}^{m-1}\frac{1}{\frac{x^{2}}{2}+2\sin ^{2}\frac{\left(
j+1/2\right) \pi }{2m}}=\sum\limits_{j=m}^{2m-1}\frac{1}{\frac{x^{2}}{2}%
+2\sin ^{2}\frac{\left( j+1/2\right) \pi }{2m}},
\end{equation*}%
and substituting $m$ by $2m$ in \eqref{sum0} with $s=\frac{x^{2}}{2}$, $\ell
=0$, and $\beta =\frac{1}{2}$, we derive that%
\begin{equation*}
\frac{2F_{2m}\left( x\right) }{xL_{2m}\left( x\right) }=\frac{U_{2m-1}\left(
\frac{x^{2}}{2}+1\right) }{T_{2m}\left( \frac{x^{2}}{2}+1\right) +1}.
\end{equation*}%
Replacing $x$ by $2x$ in the last formula proves \eqref{F2} in case of $n$
even. The proof of \eqref{F2} for odd $n$ is analogous, hence we omit it here.

\subsection{Proof of the Theorem \ref{th2}}

Comparing relations \eqref{R1} and \eqref{R2}, we deduce that
\begin{equation}
\frac{4}{5N}\sum\limits_{j=1}^{N-1}\frac{\sin ^{2}\frac{j\ell \pi }{N}}{1-%
\frac{4}{5}\sin ^{2}\frac{j\pi }{N}}=\left( -1\right) ^{\ell +1}\sqrt{5}%
F_{\ell }^{2}C_{N}^{(-1)^N}+\left( -1\right) ^{\ell }F_{2\ell },  \label{sumNeven}
\end{equation}
where $C_{N}$ is given by \eqref{C_N}.
On the other hand, we have:
\begin{align*}
\frac{4}{5N}\sum\limits_{j=1}^{N-1}\frac{\sin ^{2}\frac{j\ell \pi }{N}}{1-%
\frac{4}{5}\sin ^{2}\frac{j\pi }{N}} =
&-\frac{1}{N}\sum\limits_{j=0}^{N-1}\frac{e^{\frac{2j\pi i}{N}\cdot 0}}{-%
\frac{5}{2}+2\sin ^{2}\frac{j\pi }{N}}+\frac{1}{2N}\sum\limits_{j=0}^{N-1}%
\frac{e^{\frac{2j\ell \pi }{N}i}}{-\frac{5}{2}+2\sin ^{2}\frac{j\pi }{N}}\\&+%
\frac{1}{2N}\sum\limits_{j=0}^{N-1}\frac{e^{\frac{2j\left( N-\ell \right)
\pi i}{N}}}{-\frac{5}{2}+2\sin ^{2}\frac{j\pi }{N}}.
\end{align*}

Using \eqref{sum1} with $s=-\frac{5}{2}$, and fact that $U_{-1}\left(
x\right) =0$ for all $x$, we obtain%
\begin{equation}
\frac{4}{5N}\sum\limits_{j=1}^{N-1}\frac{\sin ^{2}\frac{j\ell \pi }{N}}{1-%
\frac{4}{5}\sin ^{2}\frac{j\pi }{N}}=\frac{U_{N-\ell -1}\left( -\frac{3}{2}%
\right) +U_{\ell -1}\left( -\frac{3}{2}\right) -U_{N-1}\left( -\frac{3}{2}%
\right) }{T_{N}\left( -\frac{3}{2}\right) -1}.  \label{sumCheb}
\end{equation}

Applying the symmetry properties \eqref{sym1} and \eqref{sym2} and relation %
\eqref{CFL1}, we obtain
\begin{equation*}
T_{n}\left( -\frac{3}{2}\right) =\left( -1\right) ^{n}\frac{L_{2n}}{2},\text{
\ \ }U_{n}\left( -\frac{3}{2}\right) =\left( -1\right) ^{n}F_{2n+2}.
\end{equation*}
Combining the above equation with \eqref{sumCheb} and \eqref{sumNeven} yields \eqref{newid1}.

\section{Further identities}

In this section we derive many identities stemming from deeper analysis of the relation between finite trigonometric sums considered in this paper and certain polynomials and integer sequences satisfying the second order recursion relations.

\subsection{Identities relating trigonometric sums and Fibonacci, Lucas, Pell and Pell-Lucas
polynomials}

Using the relations which arose in the proof of Theorem \ref{th1}, one can derive further identities, as listed in the following corollary.

\begin{corollary}
  \bigskip For any positive real number $x$, and any positive integer $m$ the following formulas hold true:%
\begin{align}
\sum_{j=0}^{2m-1}\frac{1}{x^{2}+\sin ^{2}\frac{j\pi }{2m}}& =\frac{\left(
2m-1\right) F_{2m+1}\left( 2x\right) +\left( 2m+1\right) F_{2m-1}\left(
2x\right) }{2x\left( x^{2}+1\right) F_{2m}\left( 2x\right) }+\frac{1}{x^{2}+1%
},  \label{sumeq2} \\
\sum_{j=0}^{2m-1}\frac{1}{x^{2}+4\sin ^{2}\frac{j\pi }{2m}}& =\frac{%
2mL_{2m}\left( x\right) -xF_{2m}\left( x\right) }{x\left( x^{2}+4\right)
F_{2m}\left( x\right) }+\frac{1}{x^{2}+4}.  \label{sumeq3}
\end{align}
\begin{equation}
\sum_{j=0}^{2m-1}\frac{1}{x^{2}+\sin ^{2}\frac{j\pi }{2m}}=\frac{%
mQ_{2m}\left( x\right) -xP_{2m}\left( x\right) }{x\left( x^{2}+1\right)
P_{2m}\left( x\right) }+\frac{1}{x^{2}+1}.  \label{sumeq4}
\end{equation}
\end{corollary}

\begin{proof}
If we substitute $2x$ instead of $x$ in \eqref{LF2}, we derive%
\begin{equation*}
\sum_{j=0}^{2m-1}\frac{1}{x^{2}+\sin ^{2}\frac{j\pi }{2m}}=\frac{%
2mL_{2m}\left( 2x\right) }{2x\left( x^{2}+1\right) F_{2m}\left( 2x\right) }.
\end{equation*}

Using the well-known identity (see \cite[identities (31.1) and (31.2)]{K2}%
)
\begin{equation*}
L_{n}\left( x\right) =F_{n+1}\left( x\right) +F_{n-1}\left( x\right)
\end{equation*}
and the defining recurrence formula \eqref{eq. recurr fib} for the Fibonacci polynomials, we deduce
\begin{gather*}
\sum_{j=0}^{2m-1}\frac{1}{x^{2}+\sin ^{2}\frac{j\pi }{2m}} =\frac{2m\left( F_{2m+1}\left( 2x\right) +F_{2m-1}\left(
2x\right) \right) }{2x\left( x^{2}+1\right) F_{2m}\left( 2x\right) } \\
=\frac{\left( 2m-1\right) F_{2m+1}\left( 2x\right) +\left( 2m+1\right)
F_{2m-1}\left( 2x\right) }{2x\left( x^{2}+1\right) F_{2m}\left( 2x\right) }+%
\frac{F_{2m+1}\left( 2x\right) -F_{2m-1}\left( 2x\right) }{2x\left(
x^{2}+1\right) F_{2m}\left( 2x\right) } \\
=\frac{\left( 2m-1\right) F_{2m+1}\left( 2x\right) +\left( 2m+1\right)
F_{2m-1}\left( 2x\right) }{2x\left( x^{2}+1\right) F_{2m}\left( 2x\right) }+%
\frac{2xF_{2m}\left( 2x\right) }{2x\left( x^{2}+1\right) F_{2m}\left( 2x\right) } \\
=\frac{\left( 2m-1\right) F_{2m+1}\left( 2x\right) +\left( 2m+1\right)
F_{2m-1}\left( 2x\right) }{2x\left( x^{2}+1\right) F_{2m}\left( 2x\right) }+%
\frac{1}{x^{2}+1}.
\end{gather*}%
Hence, the identity \eqref{sumeq2} obviously holds.

The identity \eqref{sumeq3} is a direct consequence of \eqref{LF2}.

Finally, the identity \eqref{sumeq4} is derived by substituting $x$ by $2x$ in \eqref{sumeq3} and using \eqref{PPL}.
\end{proof}

\begin{remark}
H. J. Seiffert (\cite{HJS}) proposed the following formula
\begin{equation*}
\sum_{j=1}^{m-1}\frac{1}{x^{2}+\sin ^{2}\frac{j\pi }{2m}}=\frac{\left(
2m-1\right) F_{2m+1}\left( 2x\right) +\left( 2m+1\right) F_{2m-1}\left(
2x\right) }{4x\left( x^{2}+1\right) F_{2m}\left( 2x\right) }-\frac{1}{2x^{2}},
\end{equation*}
which was proved by P. S. Bruckman \cite{PSB} and Koshy \cite[Example 40.2]{K2}. Identities \eqref{sumeq2}-\eqref{sumeq4} can be derived from this formula. Our aim was to provide a different proof.
\end{remark}

\subsection{Some identities involving Fibonacci and Lucas numbers}
The following two corollaries involving Fibonacci, Lucas, Pell and Pell-Lucas numbers follow from the corresponding identities involving respective polynomials.

\begin{corollary} The following identities hold true for any positive integer $m$
\begin{eqnarray}
\sum_{j=1}^{2m-1}\frac{1}{\frac{1}{4}+\sin ^{2}\frac{j\pi }{2m}} &=&\frac{4}{%
5}\frac{\left( 2m-1\right) F_{2m+1}+\left( 2m+1\right) F_{2m-1}}{F_{2m}}-%
\frac{16}{5},  \label{sum3} \\
\sum_{j=0}^{2m-1}\frac{1}{1+4\sin ^{2}\frac{j\pi }{2m}} &=&\frac{2mL_{2m}}{%
5F_{2m}},  \label{sum4} \\
\sum_{j=0}^{2m-1}\frac{1}{1+\sin ^{2}\frac{j\pi }{2m}} &=&\frac{mQ_{2m}}{%
P_{2m}}.  \label{sum5}
\end{eqnarray}
\begin{proof}
The first identity is derived by substituting $x=\frac{1}{2}$ in \eqref{sumeq2} and
using that $F_{n}=F_{n}\left( 1\right) $. The identity \eqref{sum4} follows by substituting $x=1$ in \eqref{sumeq3}. Equation \eqref{sum5} is deduced from \eqref{sumeq4} by substituting $x=1$, using
that $P_{m}=P_{m}\left( 1\right) $ and $2Q_{m}=Q_{m}\left( 1\right) $ and
performing trivial calculations.
\end{proof}
\end{corollary}
The following corollary follows by combining the above results with the identity \eqref{sum1}.
\begin{corollary}
The following identities hold true for any positive integer $m$:
  \begin{equation}
2P_{2m}P_{4m}=Q_{2m}\left( Q_{4m}-1\right), \label{eq3}
\end{equation}
  \begin{equation}
\left( 2m-1\right) F_{2m+1}+\left( 2m+1\right) F_{2m-1}+F_{2m}=\frac{%
10mF_{4m}F_{2m}}{L_{4m}-2},  \label{eq5}
\end{equation}
\begin{equation}
\left( 2m-1\right) F_{2m}F_{2m+1}+\left( 2m+1\right)
F_{2m-1}F_{2m}+F_{2m}^{2}=2mF_{4m}.  \label{eq6}
\end{equation}
\end{corollary}

\begin{proof} Setting $s=2$, $\ell =0$ and substituting $m$ by $2m$ in \eqref{sum1}, and using \eqref{CPL1} we obtain that
\begin{equation*}
\sum_{j=0}^{2m-1}\frac{1}{1+\sin ^{2}\frac{\pi j}{m}}=\frac{4m
U_{2m-1}\left( 3\right) }{T_{2m}\left( 3\right) -1} = \frac{2mP_{4m}}{
Q_{4m}-1}.
\end{equation*}%
Comparing this with \eqref{sum5} yields the identity \eqref{eq3}.

Setting $s=\frac{1}{2}$, $\ell =0$ and substituting $m$ by $2m$ in %
\eqref{sum1}, we obtain%
\begin{equation*}
\frac{1}{2m}\sum_{j=0}^{2m-1}\frac{1}{\frac{1}{2}+2\sin ^{2}\frac{\pi j}{2m}}%
=\frac{U_{2m-1}\left( \frac{3}{2}\right) }{T_{2m}\left( \frac{3}{2}\right) -1%
},
\end{equation*}%
where we use that $U_{-1}\left( x\right) =0$ for all $x$.
Rearranging the sum on the left-hand side of the above display and using relations \eqref{CFL1} we get
\begin{equation*}
\sum_{j=1}^{2m-1}\frac{1}{\frac{1}{4}+\sin ^{2}\frac{\pi j}{2m}}=\frac{%
8mF_{4m}}{L_{4m}-2}-4.
\end{equation*}%
Comparing this with \eqref{sum3} we obtain%
\begin{equation*}
\frac{4}{5}\frac{\left( 2m-1\right) F_{2m+1}+\left( 2m+1\right) F_{2m-1}}{%
F_{2m}}-\frac{16}{5}=\frac{8mF_{4m}}{L_{4m}-2}-4.
\end{equation*}%
It is now straightforward to show that the last identity implies \eqref{eq5}.
Substituting the well-known formula $L_{4m}=5F_{2m}^{2}+2$ (see e.g. \cite[%
formula 15 on p. 109]{K1}) into \eqref{eq5}, we obtain \eqref{eq6}.
\end{proof}

\subsection{Some identities involving Fibonacci hyperbolic functions}

In this section, by starting with \eqref{wu}, we deduce identities involving
Fibonacci and Lucas numbers and Fibonacci hyperbolic sine and cosine
functions. We prove the following proposition.

\begin{proposition}
\label{prop2} For a positive integer $m$, and positive integer $\ell <m/2,$ the following identities hold true:
\begin{itemize}
\item[i)] for $m$ divisible by $4$ and $\ell $ odd
$$\frac{F_{\frac{3m}{2}-3l-1}}{F_{\frac{3m}{2}}}=\frac{L_{3m-3\ell
}+L_{3\ell }}{L_{3m}-2};$$
\noindent
\item[ii)] for $m$ divisible by $4$ and $\ell $ even
$$\frac{\mathrm{cFh}\left( \frac{3m-6l-2}{4}\right) }{F_{\frac{3m%
}{2}}}=\frac{\sqrt{5}\left( F_{3m-3\ell }+F_{3\ell }\right) }{L_{3m}-2};$$
\noindent
\item[iii)] for $m$ even and not divisible by $4$ and $\ell $ even $$\frac{F_{\frac{3m}{2}-3l-1}}{\mathrm{sFh}\left( \frac{3m}{4}%
\right) }=\frac{\sqrt{5}\left( F_{3m-3\ell }+F_{3\ell }\right) }{L_{3m}-2};$$
\item[iv)] $$\frac{\mathrm{cFh}\left( \frac{3m-6l-2}{4}\right) }{\mathrm{sFh}%
\left( \frac{3m}{4}\right) }=\left\{
\begin{array}{ll}
\frac{\sqrt{5}F_{3m-3\ell }+L_{3\ell }}{\sqrt{5}F_{3m}-2}, & m\text{ {not
divisible by}}{\ 4},\ell \text{ odd} \\
\frac{L_{3m-3\ell }+\sqrt{5}F_{3\ell }}{\sqrt{5}F_{3m}-2}, & m\text{ odd, }%
\ell \text{ even}.
\end{array}%
\right.$$
\end{itemize}
\end{proposition}

\begin{proof}
We consider \eqref{wu} when $\lambda =3\log \phi $, where $\phi $ is golden
ratio. Since $\cosh (3\log \phi )=\sqrt{5}$ and $\sinh (3\log \phi )=2$, the
left-hand side of \eqref{wu} becomes%
\begin{gather*}
\frac{1}{m}\sum_{j=0}^{m-1}\frac{\cos \left( 2\ell \frac{j\pi }{m}\right) }{%
\sqrt{5}-\cos \left( 2\frac{j\pi }{m}\right) }=\frac{1}{m}\sum_{j=0}^{m-1}%
\frac{e^{2\ell \frac{ij\pi }{m}}+e^{-2\ell \frac{ij\pi }{m}}}{2\left( \sqrt{5%
}-\cos \left( 2\frac{j\pi }{m}\right) \right) } \\
=\frac{1}{m}\sum_{j=0}^{m-1}\frac{e^{2\ell \frac{ij\pi }{m}}+e^{2\ell \frac{%
i\left( m-j\right) \pi }{m}}}{2\left( \sqrt{5}-\cos \left( 2\frac{j\pi }{m}%
\right) \right) }=\frac{1}{m}\sum_{j=0}^{m-1}\frac{e^{2\ell \frac{ij\pi }{m}}%
}{\sqrt{5}-\cos \left( 2\frac{j\pi }{m}\right) }=\frac{1}{m}\sum_{j=1}^{m-1}%
\frac{e^{2\pi i\frac{j\ell }{m}}}{\sqrt{5}-1+2\sin ^{2}\frac{\pi j}{m}}.
\end{gather*}%
By applying \eqref{sum1} and \eqref{CFL4} we get
\begin{eqnarray}
\frac{1}{m}\sum_{j=1}^{m-1}\frac{e^{2\pi i\frac{j\ell }{m}}}{\sqrt{5}%
-1+2\sin ^{2}\frac{\pi j}{m}} &=&\frac{U_{m-\ell -1}\left( \sqrt{5}\right)
+U_{\ell -1}\left( \sqrt{5}\right) }{T_{m}\left( \sqrt{5}\right) -1}
\label{eq1} \\
&=&\left\{
\begin{array}{ll}
\frac{\sqrt{5}\left( F_{3m-3\ell }+F_{3\ell }\right) }{2L_{3m}-4}, & m,\ell
\text{ even} \\
\frac{L_{3m-3\ell }+L_{3\ell }}{2L_{3m}-4}, & m\text{ even, }\ell \text{ odd}
\\
\frac{\sqrt{5}F_{3m-3\ell }+L_{3\ell }}{2\sqrt{5}F_{3m}-4}, & m,\ell \text{
odd} \\
\frac{L_{3m-3\ell }+\sqrt{5}F_{3\ell }}{2\sqrt{5}F_{3m}-4}, & m\text{ odd, }%
\ell \text{ even}%
\end{array}%
\right. .  \notag
\end{eqnarray}%
The right-hand side of \eqref{wu} for $\lambda =3\log \phi $ can be
expressed by Fibonacci hyperbolic functions (see \eqref{Fh1}) as
\begin{equation*}
\frac{\cosh \left[ \left( \frac{3m}{2}-3\ell \right) \log \phi \right] }{%
2\sinh \frac{3m\log \phi }{2}}=\frac{\mathrm{cFh}\left( \frac{3m-6\ell -2}{4}%
\right) }{2\ \mathrm{sFh}\left( \frac{3m}{4}\right) }.
\end{equation*}%
Using \eqref{Fh2}, we get following special cases%
\begin{equation}
\frac{\mathrm{cFh}\left( \frac{3m-6\ell -2}{4}\right) }{2\ \mathrm{sFh}%
\left( \frac{3m}{4}\right) }=\left\{
\begin{array}{ll}
\frac{F_{\frac{3m}{2}-3l-1}}{2F_{\frac{3m}{2}}}, & m\equiv 0\pmod{4}\text{ and }\ell \text{ odd} \\
\frac{\mathrm{cFh}\left( \frac{3m-6l-2}{4}\right) }{2F_{\frac{3m}{2}}}, &
m\equiv 0\pmod{4}\text{ and }\ell \text{ even} \\
\frac{F_{\frac{3m}{2}-3l-1}}{2\mathrm{sFh}\left( \frac{3m}{4}\right) }, & {m}%
\text{ {even, not divisible by }}{4\ }\text{{and}}{\ \ell \ }\text{{even}}%
\end{array}%
\right. .  \label{eq2}
\end{equation}
Combining \eqref{eq1} and \eqref{eq2} we see that i)-iv) hold.
\end{proof}

\begin{remark}
If we substitute $\lambda =\log \phi $ in formula \eqref{wu} and using \eqref{CFL3}, we obtain the
same result as in Proposition \ref{prop2} with $3m$ replaced by $m$ and $3\ell $
replaced by $\ell $.
\end{remark}

\subsection{Identities involving Bejaia and Pisa numbers}

In this section, we derive some further identities involving Chebyshev polynomials, Fibonacci and Lucas numbers and the so-called Bejaia and Pisa numbers which are generalizations of the bisected
Fibonacci and Lucas numbers respectively. For a positive integer $N$, Bejaia and Pisa numbers  $\mathcal{B}%
_{\ell }(N)$ and $\mathcal{P}_{\ell }(N)$, respectively were introduced by Chair \cite[relations (17) and (18)]{NC1} and defined as:
\begin{eqnarray*}
\mathcal{B}_{\ell }(N) &=&\frac{1}{\sqrt{N\left( N-4\right) }}\left( \left(
\frac{N-2+\sqrt{N\left( N-4\right) }}{2}\right) ^{\ell }-\left( \frac{N-2-%
\sqrt{N\left( N-4\right) }}{2}\right) ^{\ell }\right) , \\
\mathcal{P}_{\ell }(N) &=&\left( \frac{N-2+\sqrt{N\left( N-4\right) }}{2}%
\right) ^{\ell }+\left( \frac{N-2-\sqrt{N\left( N-4\right) }}{2}\right)
^{\ell }.
\end{eqnarray*}
Those numbers are related to a trigonometric sum similar to \eqref{sumNeven} and occur in the (exact) formula for the two-point resistance of the complete graph minus $N$ edges of the opposite vertices, for $N$ odd. Namely, Chair in \cite[relation (25)]{NC1} derived the following formula, valid for $\ell <\frac{N+1}{2}$:
\begin{equation}
\frac{4}{N}\sum\limits_{n=0}^{\frac{N-1}{2}}\frac{\sin ^{2}\left( \frac{%
2n\ell \pi }{N}\right) }{N-4\sin ^{2}\left( \frac{n\pi }{N}\right) }=\frac{%
\mathcal{B}_{2\ell }(N)}{2}-\frac{\sqrt{N\left( N-4\right) }}{2}\mathcal{B}%
_{\ell }^{2}(N)D_{N},  \label{BN1}
\end{equation}%
where%
\begin{equation*}
D_{N}=\left(1-\left( \frac{N-2-\sqrt{N\left( N-4\right) }}{2}\right) ^{N}\right) \left(
1+\left( \frac{N-2-\sqrt{N\left( N-4\right) }}{2}\right) ^{N}\right)^{-1}.
\end{equation*}
As above, it is possible to express the trigonometric sum on the left-hand side of \eqref{BN1} in terms of Chebyshev polynomials and deduce the following corollary.
\begin{corollary}
For positive integers $N$ and $\ell <\frac{N+1}{2}$ the following identity holds true
\begin{multline}\label{eq. odd N chair id}
  \frac{4}{N}\sum\limits_{n=0}^{\frac{N-1}{2}}\frac{\sin ^{2}\left( \frac{%
2n\ell \pi }{N}\right) }{N-4\sin ^{2}\left( \frac{n\pi }{N}\right) }= \\=\frac{U_{N-2\ell
-1}\left( -\frac{N}{2}+1\right) +U_{2\ell -1}\left( -\frac{N}{2}+1\right)
-U_{N-1}\left( -\frac{N}{2}+1\right) }{2T_{N}\left( -\frac{N}{2}+1\right) -2}.
\end{multline}
\end{corollary}
\begin{proof}
The proof is analogous to the proof of equation \eqref{sumCheb}, with the difference that in this case, we take $s=-N/2$ in \eqref{sum1}. We omit the details.
\end{proof}

By specifying $N$ in \eqref{eq. odd N chair id}, and combining it with \eqref{BN1}, one can produce many interesting identities. In the following examples we present two possibilities.

\begin{example} Let us take $N=9$ and $\ell\in\{0,1,2,3,4\}$ in \eqref{eq. odd N chair id}. Then, \eqref{eq. odd N chair id} and  \eqref{BN1} for $N=9$ yield the identity
\begin{equation*}
\frac{%
\mathcal{B}_{2\ell }(9)}{2}-\frac{3\sqrt{5}}{2}\mathcal{B}%
_{\ell }^{2}(9)D_{9}=-\frac{1}{2}\left( \frac{U_{8}\left( -\frac{7}{2}\right) - U_{8-2\ell
}\left( -\frac{7}{2}\right) -U_{2\ell -1}\left( -\frac{7}{2}\right)}{T_{9}\left( -\frac{7}{2}\right) -1} \right).
\end{equation*}
Using symmetry properties \eqref{sym1} and \eqref{sym2}, together with  relations \eqref{CFL2} we get
\begin{equation*}
\frac{%
\mathcal{B}_{2\ell }(9)}{2}-\frac{3\sqrt{5}}{2}\mathcal{B}%
_{\ell }^{2}(9)D_{9}=\frac{1}{3}\frac{%
-F_{36-8\ell }+F_{8\ell }+F_{36}}{L_{36}+2}.
\end{equation*}
Therefore, we obtain the following identity
\begin{equation*}
2\left( F_{8\ell }+F_{36}-F_{36-8\ell }\right) =\left( 3\mathcal{B}_{2\ell
}(9)-9\sqrt{5}\mathcal{B}_{\ell }^{2}(9)D_{9}\right) \left( L_{36}+2\right).
\end{equation*}
\end{example}

\begin{example} In this example we take $N=L_{2k}+2$, where $k$ is such that $L_{2k}$ is odd and take $\ell <\frac{N+1}{2}$. By reasoning analogously as in the previous example, using special values \eqref{chebluc} and \eqref{chebfib} of Chebyshev polynomials, we derive the following identity
\begin{equation*}
\frac{F_{2Nk}+F_{2\ell k}-F_{2\left( N-2\ell \right) k}}{\left(
L_{2Nk}+2\right) F_{2k}}=\frac{1}{2}\left( \mathcal{B}_{2\ell }(N)-\sqrt{%
N\left( N-4\right) }\mathcal{B}_{\ell }^{2}(N)D_{N}\right).
\end{equation*}%
\end{example}

\end{document}